\documentclass[reqno,12pt]{amsart}
\usepackage{amssymb}
\usepackage{amsmath}
\usepackage{amsthm}
\usepackage{mathtools}
\usepackage{amsfonts}
\usepackage{xcolor}
\usepackage{graphicx}
\usepackage{tikz}
\usepackage{comment}
\usepackage{microtype}
\usepackage{enumitem}
\usepackage[
    colorlinks = true, 
    linkcolor={blue},
	citecolor={blue},
	urlcolor={black!30!blue}
]{hyperref}
\usepackage[alphabetic, initials]{amsrefs}
\renewcommand{\MR}[1]{}

\title[Primitive approximation for fixed frequencies]{Primitive inhomogeneous approximation for fixed non-singular frequencies}

\author[X.\ Wang]{Xueyin Wang}
\address{[X.\ Wang] Department of Mathematics, Texas A\&M University, College Station, TX 77843, USA}
\email{\href{mailto:xueyin@tamu.edu}{xueyin@tamu.edu}}

\theoremstyle{plain}
\newtheorem{theorem}{Theorem}[section]

\newtheorem{lemma}[theorem]{Lemma}
\newtheorem{proposition}[theorem]{Proposition}

\theoremstyle{definition}

\newtheorem{remark}[theorem]{Remark}

\numberwithin{equation}{section}

\begin{document}
\begin{abstract}
	We prove a high-dimensional primitive inhomogeneous Diophantine approximation result for fixed non-singular simultaneous frequencies. The frequency class is explicit and has full Lebesgue measure.
\end{abstract}

\maketitle	

\section{Introduction}
In this paper, we prove a higher-dimensional analogue of primitive inhomogeneous approximation along a \emph{fixed} frequency $\alpha$. In dimension one, Jitomirskaya and Liu \cite{MR3996322} proved that for every irrational $\alpha\in\mathbb{R}\setminus\mathbb{Q}$, for Lebesgue almost every $\gamma\in\mathbb{R}$,
\begin{equation*}
    \liminf_{n\to\infty} n \min_{\substack{m\in\mathbb{Z}\\ \gcd(n,m)=1}} |\gamma-n\alpha+m| = 0.
\end{equation*}
In higher dimensions, the natural analogue of the coprime condition is the joint primitive condition
\begin{equation*}
    \gcd(n,m_1,\ldots,m_d)=1,
\end{equation*}
which means that the integer vector $(n,m_1,\ldots,m_d)\in\mathbb{Z}^{d+1}$ is primitive.

For $x\in\mathbb{R}^{d}$, define
\begin{equation*}
    \|x\|_{\mathbb{T}^{d},\infty} = \min_{k\in\mathbb{Z}^{d}}\|x-k\|_{\infty}.
\end{equation*}
In particular, we write $\|x\|_{\mathbb{T}}$ when $d=1$ for simplicity. We define the non-singular frequency set $\Lambda^{d}\subseteq\mathbb{R}^{d}$ as follows. We say that $\alpha=(\alpha_1,\ldots,\alpha_d)\in\Lambda^{d}$ if $1,\alpha_1,\ldots,\alpha_d$ are linearly independent over $\mathbb{Q}$ and there exist $c_0>0$ and a sequence $N_j\to\infty$ such that
\begin{equation}\label{Gintro}
    \min_{1\leqslant n\leqslant N_j}\|n\alpha\|_{\mathbb{T}^{d},\infty}\geqslant c_0N_j^{-1/d}.
\end{equation}
The set $\Lambda^{d}$ has full Lebesgue measure (see \hyperref[fullLebLambda]{Theorem~\ref{fullLebLambda}}).

Our main result is as follows.

\begin{theorem}\label{mainthm}
    Let $\alpha\in \Lambda^{d}$ with $d\geqslant 2$. Then for Lebesgue almost every $\gamma \in \mathbb{R}^d$,
    \begin{equation}\label{mainclaim}
        \liminf_{n\to\infty}
        n^{1/d} \min_{\substack{m\in\mathbb{Z}^d\\ \gcd(n,m_1,\ldots,m_d)=1}} \|\gamma-n\alpha+m\|_\infty= 0.
    \end{equation}
\end{theorem}

We first explain why the assumption $\alpha\in\Lambda^{d}$ is natural. Without the primitivity constraint, fixed-frequency inhomogeneous approximation has a long history beginning with Kurzweil \cite{MR73654}. In the case $d=1$, non-singularity is equivalent to irrationality by the classical estimate $\|q_{j}\alpha\|_{\mathbb{T}} \geqslant 1/(2q_{j+1})$ for the best approximation denominators $\{q_{j}\}$, which is consistent with the one-dimensional theorem of Jitomirskaya and Liu \cite{MR3996322}. In higher dimensions $d\geqslant2$, however, rational independence alone is not the correct fixed-frequency hypothesis. The appropriate substitute is the non-singularity condition from simultaneous Diophantine approximation. This point of view appears in work of Kim \cite{MR2335077}, Shapira \cite{MR3048195}, and more recently Beresnevich, Datta, Ghosh, and Ward \cite{MR4834219} on shrinking targets for torus actions.

Indeed, Shapira \cite{MR3048195} proved that for any fixed non-singular frequency $\alpha\in\Lambda^{d}$ and Lebesgue almost every $\gamma\in\mathbb{R}^{d}$,
\begin{equation}\label{secclaim}
    \liminf_{n\to\infty} n^{1/d}\min_{m\in\mathbb{Z}^{d}}\|\gamma-n\alpha+m\|_{\infty}=0.
\end{equation}
On the other hand, a question of Tseng \cite{MR2379490} asked whether \eqref{secclaim} holds for every rationally independent frequency. This is false. Galatolo and Peterlongo \cite{MR2600766} constructed a rationally independent but singular vector in $\mathbb{R}^{2}$ giving a negative answer to this question. Thus one cannot expect a fixed-frequency theorem such as \eqref{mainclaim} to hold for all rationally independent frequencies. Since this obstruction already appears without the primitive restriction, it is natural to formulate the primitive fixed-frequency theorem on the explicit full-measure class $\Lambda^{d}$.

We now compare \hyperref[mainthm]{Theorem~\ref{mainthm}} with known results on primitive approximation. In dimension one, Jitomirskaya and Liu \cite{MR3996322} proved the fixed-frequency coprime theorem for every irrational frequency. In higher dimensions, Dani, Laurent, and Nogueira \cite{MR3318261} introduced a flexible class of primitivity constraints associated with partitions of the coordinate directions and proved a doubly-metric Khintchine--Groshev theorem for primitive points. In the simultaneous setting considered here, their result implies that \eqref{mainclaim} holds for Lebesgue almost every pair $(\alpha,\gamma)$. By Fubini's theorem, this gives an abstract full-measure set of frequencies $\alpha$ for which \eqref{mainclaim} holds for Lebesgue almost every $\gamma$. However, this argument does not identify a concrete arithmetic class of admissible frequencies, and it does not give a fixed-frequency proof for every prescribed $\alpha$ in an explicit class.

More recently, Allen and Ram\'irez \cite{MR4832558} strengthened this framework. They proved that for any fixed $\gamma$, \eqref{mainclaim} holds for almost every $\alpha$. They also pointed out that the complementary problem, where $\alpha$ is fixed and one studies metric statements in $\gamma$, is natural. \hyperref[mainthm]{Theorem~\ref{mainthm}} addresses this fixed $\alpha$ direction at the scale $n^{-1/d}$ for the explicit full-measure class $\Lambda^{d}$ of non-singular simultaneous frequencies.

The novelty of \hyperref[mainthm]{Theorem~\ref{mainthm}} is therefore twofold. First, it is a fixed-frequency result: the frequency $\alpha$ may be prescribed in advance, provided $\alpha\in\Lambda^{d}$. Second, it shows that the primitive restriction does not destroy the critical shrinking-target scale $n^{-1/d}$ for every non-singular simultaneous frequency.

We now explain the idea of the proof. The proof consists of two main parts: a primitive discrepancy estimate along the good scales and a shrinking-target construction which produces a full-measure set of admissible targets $\gamma$. The primitive discrepancy estimate is proved in Section \ref{sec:discrepancy}. The proof is based on an elementary sieve. Fixed congruence conditions are handled by Weyl equidistribution, small prime divisors are removed by the M\"obius function sums, and the contribution of large prime divisors is controlled by the packing estimate coming from the good-scale condition \eqref{Gintro}. 

The construction of the full-measure set of targets is carried out in Section \ref{sec:targets}. At each good scale $N_{j}$, we place small boxes of suitable side length $r_{j}$ around the primitive orbit points $\{n\alpha\}$, and then show that the union of these boxes has asymptotically positive measure. We then prove a quasi-independence estimate between different good scales and apply a second-moment Borel--Cantelli lemma to obtain a full measure set of the targets $\gamma$.

Throughout this paper, we use the following notation. We write $A=O(B)$ if there exists a constant $C>0$ such that $|A|\leqslant C|B|$. Moreover, we write
$A=O_{r}(B)$  if the constant depends on the parameter $r$. We write $A=o(B)$ if $A/B\to0$ in the relevant limiting process, and write $A=o_{r}(B)$ if the convergence may depend on the fixed parameter $r$. Finally, $A\sim B$ means that $A/B\to1$.

\section{Preliminaries}

For $x\in\mathbb{R}$, $\lfloor x\rfloor$ denotes the greatest integer not exceeding $x$. We identify $\mathbb{T}^{d}$ with $[0,1)^{d}$. For $x \in \mathbb{R}$, let
\begin{equation*}
    \{x\} = x - \lfloor x \rfloor \in [0,1).
\end{equation*}
For any $A\subseteq [0,1)^{d}$ and $x\in [0,1)^{d}$, define
\begin{equation*}
    \operatorname{dist}_{\infty}(x, A)= \inf_{y\in A} \|x-y\|_{\infty}.
\end{equation*}

Let $(X,\operatorname{dist})$ be a metric space. We say that a set $\{a_{n}\}_{n\in\mathcal{N}}\subseteq X$ is $s$-separated if
\begin{equation*}
    \operatorname{dist}(a_{n_{1}},a_{n_{2}})\geqslant s
\end{equation*}
for any distinct $n_{1},n_{2}\in\mathcal{N}$.

The separation property in \eqref{Gintro} provides a discrepancy estimate in the cube.

\begin{lemma}\label{packing}
Assume \eqref{Gintro} holds. Then for every $j$, the points $\{n\alpha\}, 1 \leqslant n \leqslant N_{j}$ are $c_{0}N_{j}^{-1/d}$-separated in $(\mathbb{T}^{d}, \|\cdot\|_{\mathbb{T}^{d},\infty})$. Consequently, there exists a constant $C = C(d,c_{0})$ such that for every axis-parallel cube $Q \subseteq \mathbb{T}^{d}$ of side length $\rho$,
\begin{equation}\label{packingbound}
    \#\big\{ 1 \leqslant n \leqslant N_{j} : \{n\alpha\} \in Q \big\} \leqslant C (1 + \rho^{d} N_{j}).
\end{equation}
\end{lemma}

\begin{proof}
If $1 \leqslant n_{1} < n_{2} \leqslant N_{j}$, then by \eqref{Gintro},
\begin{equation*}
    \|\{n_{1}\alpha\} - \{n_{2}\alpha\}\|_{\mathbb{T}^{d},\infty} = \|(n_{2} - n_{1})\alpha\|_{\mathbb{T}^{d},\infty} \geqslant c_{0} N_{j}^{-1/d},
\end{equation*}
which implies that they are $c_{0}N_{j}^{-1/d}$-separated.

Let $\Delta = c_{0} N_{j}^{-1/d}$. A cube of side length $\rho$ contains at most
\begin{equation*}
    C_{d} \bigg(1 + \frac{\rho}{\Delta}\bigg)^{d} \leqslant C(d,c_{0}) (1 + \rho^{d} N_{j})
\end{equation*}
points. This gives \eqref{packingbound}.
\end{proof}

We shall also use the following second moment Borel--Cantelli Lemma. For completeness, we provide a proof in Appendix \ref{ABC}.

\begin{lemma}[\cites{MR45327,MR161355}]\label{SMBC}
Let $(X,\mu)$ be a probability space, and let $F_{1},F_{2},\ldots$ be measurable subsets of $X$. Assume that
\begin{equation*}
    S_{L}\coloneq \sum_{\ell=1}^{L}\mu(F_{\ell})\to\infty\qquad \text{as }L\to\infty,
\end{equation*}
and
\begin{equation}\label{SMBCcondition}
    \sum_{a,\ell=1}^{L}\mu(F_{a}\cap F_{\ell})=(1+o(1))S_{L}^{2}\qquad \text{as }L\to\infty.
\end{equation}
Then
\begin{equation*}
    \mu\bigg(\limsup_{\ell\to\infty}F_{\ell}\bigg)=1.
\end{equation*}
\end{lemma}

\section{Primitive discrepancy estimates}
\label{sec:discrepancy}
Fix $b = (b_{1}, \ldots, b_{d}) \in \mathbb{Z}^{d}$. Define the primitive set
\begin{equation*}
    \mathcal{P}_{b} = \{n \geqslant 1 : \gcd\big(n, \lfloor n\alpha_{1}\rfloor - b_{1}, \ldots, \lfloor n\alpha_{d}\rfloor - b_{d}\big) = 1 \}.
\end{equation*}
A bounded set $A \subseteq [0,1)^{d}$ is called Jordan measurable if its boundary $\partial A$ has Lebesgue measure zero.

\begin{theorem}\label{counting}
    Let $\alpha \in \Lambda^{d}$ with $d \geqslant 2$. Let $A \subseteq [0,1)^{d}$ be Jordan measurable. Then, along the good scales $N_{j}$,
\begin{equation}\label{countingformula}
    \#\big\{ 1 \leqslant n \leqslant N_{j} : n \in \mathcal{P}_{b}, \ \{n\alpha\} \in A \big\} = \big(\delta_{d}\operatorname{Leb}(A) + o(1)\big)N_{j},
\end{equation}
where
\begin{equation}\label{deltad}
    \delta_{d} = \prod_{p \ \text{is prime}} \bigg(1 - \frac{1}{p^{d+1}}\bigg) = \frac{1}{\zeta(d+1)}.
\end{equation}
\end{theorem}

\begin{remark}
    \hyperref[counting]{Theorem~\ref{counting}} is a higher-dimensional analogue of  \cite{MR3996322}*{Theorem 2.1} in dimension one.
\end{remark}

\begin{proof}
Let $r \geqslant 1$ be square-free. Define
\begin{equation*}
    G_{b}(n) = \gcd(n, \lfloor n\alpha_{1}\rfloor - b_{1}, \ldots, \lfloor n\alpha_{d}\rfloor - b_{d}).
\end{equation*}
Define
\begin{equation}\label{defDr}
    D_{r}(A,N) = \#\big\{ 1 \leqslant n \leqslant N : r \mid G_{b}(n), \ \{n\alpha\} \in A \big\}.
\end{equation}
We first compute $D_{r}(A,N)$ for fixed $r$.

Write $n = r\ell$. For each $1 \leqslant i \leqslant d$, let $c_{i} \in \{0, \ldots, r-1\}$ be such that
\begin{equation*}
    b_{i}\equiv c_{i}\pmod r.
\end{equation*}
Since
\begin{equation*}
    \ell \alpha_{i} = \lfloor \ell \alpha_{i}\rfloor + \{\ell \alpha_{i}\},
\end{equation*}
multiplying by $r$ and then taking the integer part on both sides gives
\begin{equation*}
    \lfloor r\ell\alpha_{i}\rfloor = r\lfloor \ell\alpha_{i}\rfloor + \lfloor r\{\ell\alpha_{i}\}\rfloor,
\end{equation*}
which means
\begin{equation*}
    \lfloor r\ell\alpha_{i}\rfloor \equiv \lfloor r\{\ell\alpha_{i}\}\rfloor \pmod{r}.
\end{equation*}
Thus the condition $r \mid \lfloor r\ell\alpha_{i}\rfloor - b_{i}$ is equivalent to $\lfloor r\{\ell\alpha_{i}\}\rfloor = c_{i}$, which is further equivalent to
\begin{equation*}
    \{\ell\alpha_{i}\} \in \bigg[\frac{c_{i}}{r}, \frac{c_{i}+1}{r}\bigg).
\end{equation*}
Under the condition $\lfloor r\{\ell\alpha_{i}\}\rfloor = c_{i}$, one has
\begin{equation}\label{congruence}
    \begin{split}
        \{r\ell\alpha_{i}\} &= \{r\lfloor\ell \alpha_{i} \rfloor + r \{\ell \alpha_{i}\}\} \\
        &= \{r \{\ell \alpha_{i}\}\} \\
        &= r\{\ell \alpha_{i}\} - \lfloor r\{\ell\alpha_{i}\}\rfloor \\
        &= r\{\ell\alpha_{i}\} - c_{i}.
    \end{split}
\end{equation}

Let $Q_{r,c}$ be the cube
\begin{equation*}
    Q_{r,c}\coloneq \prod_{i=1}^{d}\bigg[\frac{c_i}{r},\frac{c_i+1}{r}\bigg).
\end{equation*}
Define 
\begin{equation*}
    A_{r,b}\coloneq \{y\in Q_{r,c}:(ry_{1}-c_{1}, \ldots, ry_{d}-c_{d}) \in A \}.
\end{equation*}
Then $A_{r,b}$ is Jordan measurable, and
\begin{equation}\label{Armeasure}
    \operatorname{Leb}(A_{r,b}) = \frac{\operatorname{Leb}(A)}{r^{d}}.
\end{equation}
Moreover, by \eqref{congruence} and \eqref{defDr},
\begin{equation*}
    D_{r}(A,N) = \#\big\{ 1 \leqslant \ell \leqslant \lfloor N/r \rfloor : \{\ell\alpha\} \in A_{r,b} \big\}.
\end{equation*}
Since $1, \alpha_{1}, \ldots, \alpha_{d}$ are linearly independent over $\mathbb{Q}$, the sequence $\{\ell\alpha\}$ is uniformly distributed in $\mathbb{T}^{d}$. Therefore, by Weyl's equidistribution theorem and \eqref{Armeasure}, for fixed $r$,
\begin{equation}\label{Drasymp}
    D_{r}(A,N) = \frac{N}{r} \cdot \operatorname{Leb}(A_{r,b}) + o(N) = \frac{\operatorname{Leb}(A)}{r^{d+1}}N + o(N).
\end{equation}

Let
\begin{equation}\label{PRdef}
    P_{R} = \prod_{\substack{p \leqslant R \\ p \text{ is prime}}} p.
\end{equation}
The number of integers $1 \leqslant n \leqslant N$ such that $\{n\alpha\} \in A$ and no prime $p \leqslant R$ divides all of $n, \lfloor n\alpha_{1}\rfloor - b_{1}, \ldots, \lfloor n\alpha_{d}\rfloor - b_{d}$ is given by
\begin{equation*}
    \begin{split}
        S_{R}(A,N) &=\# \{1\leqslant n\leqslant N: \{n\alpha\}\in A, \gcd(G_{b}(n), P_{R})=1\}\\
        &=\sum_{\substack{1\leqslant n\leqslant N\\ \{n\alpha\}\in A}}\sum_{\substack{r\mid P_{R}\\ r\mid G_{b}(n)}} \mu(r),
    \end{split}
\end{equation*}
where $\mu(\cdot)$ is the M\"{o}bius function,
\begin{equation*}
    \mu(r) = 
    \begin{cases}
        1,& \text{ if } r = 1,\\
        (-1)^{k},& \text{ if } r \text{ is the product of } k \text{ distinct primes},\\
        0,& \text{ if } r \text{ is divisible by the square of a prime},
    \end{cases}
\end{equation*}
and we use the property
\begin{equation*}
    \sum_{\substack{r\mid P_{R}\\ r\mid G_{b}(n)}} \mu(r)=\begin{cases}
        1,& \text{ if } \gcd(G_{b}(n), P_{R})=1,\\
        0,&\text{ if } \gcd(G_{b}(n), P_{R})>1.
    \end{cases}
\end{equation*}
By \eqref{defDr}, we can rewrite
\begin{equation}\label{IEP}
    S_{R}(A,N) =\sum_{r \mid P_{R}} \mu(r)D_{r}(A,N).
\end{equation}
Using \eqref{Drasymp} and \eqref{IEP}, for fixed $R$,
\begin{equation*}
    \begin{split}
        S_{R}(A,N)&=\sum_{r\mid P_{R}}\mu(r) \bigg(\frac{\operatorname{Leb}(A)}{r^{d+1}}N + o(N) \bigg)\\
        &=\operatorname{Leb}(A) N \sum_{r\mid P_{R}}\frac{\mu(r)}{r^{d+1}} + o_{R}(N).
    \end{split}
\end{equation*}
Here the notation $o_R(N)$ indicates that the implicit error term depends on $R$. Now we estimate $\sum_{r\mid P_{R}}\frac{\mu(r)}{r^{d+1}}$. Note that $P_{R}$ is square-free according to the definition \eqref{PRdef}, thus every $r\mid P_{R}$ is the product of distinct primes, which means for such $r$,
\begin{equation*}
    \mu(r)=(-1)^{\omega(r)},
\end{equation*}
where $\omega(r)$ is the number of distinct prime factors of $r$. Thus the sum can be expressed as the standard Euler product,
\begin{equation*}
    \sum_{r\mid P_{R}}\frac{\mu(r)}{r^{d+1}}= \prod_{\substack{p \leqslant R \\ p \text{ is prime}}} \bigg(1 - \frac{1}{p^{d+1}}\bigg).
\end{equation*}
Therefore
\begin{equation}\label{SRasymp}
    S_{R}(A,N) = \operatorname{Leb}(A)N \prod_{\substack{p \leqslant R \\ p \text{ is prime}}} \bigg(1 - \frac{1}{p^{d+1}}\bigg) + o_{R}(N).
\end{equation}

It remains to estimate the contribution of large primes. Let $T_{R}(N_{j})$ denote the number of integers $1 \leqslant n \leqslant N_{j}$ for which there exists a prime $p > R$ satisfying
\begin{equation}\label{largeprimecondition}
    p \mid n, \qquad p \mid \lfloor n\alpha_{i}\rfloor - b_{i} \quad \text{for every } 1 \leqslant i \leqslant d.
\end{equation}
We claim that
\begin{equation}\label{TRbound}
    \limsup_{j\to\infty} \frac{T_{R}(N_{j})}{N_{j}} \leqslant C_{\alpha,d}R^{1-d}.
\end{equation}
Indeed, fix a prime $p > R$ and write $n = p\ell$. Let $c_{i}'(p) \in \{0, \ldots, p-1\}$ be such that
\begin{equation*}
    b_{i}\equiv c_{i}'(p)\pmod p.
\end{equation*}
From
\begin{equation*}
    \lfloor p\ell\alpha_{i}\rfloor = p\lfloor \ell\alpha_{i}\rfloor + \lfloor p\{\ell\alpha_{i}\}\rfloor,
\end{equation*}
condition \eqref{largeprimecondition} implies
\begin{equation*}
    \{\ell\alpha_{i}\} \in \bigg[\frac{c_{i}'(p)}{p}, \frac{c_{i}'(p)+1}{p}\bigg) \quad \text{for every}\ 1 \leqslant i \leqslant d.
\end{equation*}
Thus $\{\ell\alpha\}$ lies in an axis-parallel cube of side length $p^{-1}$. Since
\begin{equation*}
    1 \leqslant \ell \leqslant \frac{N_{j}}{p} \leqslant N_{j},
\end{equation*}
\hyperref[packing]{Lemma~\ref{packing}} gives
\begin{equation}\label{perpbad}
    \#\bigg\{ 1 \leqslant \ell \leqslant \frac{N_{j}}{p} : p \mid \lfloor p\ell\alpha_{i}\rfloor - b_{i} \text{ for every } i \bigg\} \leqslant C_{\alpha,d}\bigg(1 + \frac{N_{j}}{p^{d}}\bigg).
\end{equation}
By summing \eqref{perpbad} over primes $R<p \leqslant N_{j}$, we get
\begin{equation*}
    T_{R}(N_{j}) \leqslant C_{\alpha,d} \sum_{\substack{R<p\leqslant N_{j} \\ p\text{ is prime}}} \bigg(1 + \frac{N_{j}}{p^{d}}\bigg).
\end{equation*}
Therefore,
\begin{equation*}
    \frac{T_{R}(N_{j})}{N_{j}} \leqslant C_{\alpha,d}\frac{\pi(N_{j})}{N_{j}} + C_{\alpha,d} \sum_{\substack{p>R\\ p\ \text{is prime}}} \frac{1}{p^{d}},
\end{equation*}
where $\pi(\cdot)$ is the prime-counting function,
\begin{equation*}
    \pi(N)=\#\{1<p\leqslant N: p\ \text{is prime}\}.
\end{equation*}
Since $d \geqslant 2$,
\begin{equation*}
    \sum_{\substack{p>R\\ p\ \text{is prime}}} \frac{1}{p^{d}} \leqslant \sum_{m>R} \frac{1}{m^{d}} \leqslant C_{d} R^{1-d}.
\end{equation*}
Also, it is well-known that $\pi(N_{j})/N_{j} \to 0$ as $j\to\infty$. Hence, \eqref{TRbound} follows.

Let
\begin{equation*}
    P_{b}(A,N) = \#\big\{ 1 \leqslant n \leqslant N :  \{n\alpha\} \in A, \ G_{b}(n)=1 \big\}.
\end{equation*}
Every primitive $n$ is counted by $S_{R}(A,N)$. Conversely, a non-primitive $n$ counted by $S_{R}(A,N)$ must have a common prime divisor $p > R$ of $n, \lfloor n\alpha_{1}\rfloor - b_{1}, \ldots, \lfloor n\alpha_{d}\rfloor - b_{d}$.
Thus,
\begin{equation}\label{diffbound}
    0 \leqslant S_{R}(A,N_{j}) - P_{b}(A,N_{j}) \leqslant T_{R}(N_{j}).
\end{equation}
Combining \eqref{SRasymp}, \eqref{TRbound}, and \eqref{diffbound}, we obtain
\begin{equation*}
    \limsup_{j\to\infty} \bigg| \frac{P_{b}(A,N_{j})}{N_{j}} - \operatorname{Leb}(A) \prod_{\substack{p \leqslant R \\ p \text{ is prime}}} \bigg(1 - \frac{1}{p^{d+1}}\bigg) \bigg| \leqslant C_{\alpha,d}R^{1-d}.
\end{equation*}
Letting $R \to \infty$ gives \eqref{countingformula}.
\end{proof}

\section{Shrinking targets along good scales}
\label{sec:targets}
Let $\alpha\in \Lambda^{d}$.
Fix $b \in \mathbb{Z}^{d}$. Let $0 < \tau < c_{0}/8$. For each good scale $N_{j}$, define
\begin{equation}\label{rjdef}
    r_{j} = \tau N_{j}^{-1/d}.
\end{equation}
For $x \in [0,1)^{d}$, define the box
\begin{equation*}
    B_{j}(x) = [0,1)^{d} \cap \prod_{i=1}^{d}(x_{i} - r_{j}, x_{i} + r_{j}).
\end{equation*}

Define
\begin{equation}\label{Ejdef}
    E_{j}^{b}(\tau) = \bigcup_{\substack{1 \leqslant n \leqslant N_{j}\\ n \in \mathcal{P}_{b}}} B_{j}(\{n\alpha\}) \subseteq [0,1)^{d}.
\end{equation}
We call $E_{j}^{b}(\tau)$ the ``targets'' following the terminology used in \cite{MR4834219}. 
\hyperref[fig:targets2d]{Figure~\ref{fig:targets2d}} illustrates an interior target, a target intersecting one side of the boundary, and a target intersecting two sides near a corner.

\begin{figure}[htbp]
\centering
\begin{tikzpicture}[line cap=round,line join=round]

\draw[thick] (0cm,0cm) rectangle (4cm,4cm);

\node[below left] at (0cm,0cm) {$0$};
\node[below] at (4cm,0cm) {$1$};
\node[left] at (0cm,4cm) {$1$};
\node[below] at (2cm,0cm) {$x_{1}$};
\node[left] at (0cm,2cm) {$x_{2}$};
\node[above] at (2cm,4cm) {$[0,1)^{2}$};




\fill[blue!18] (0.75cm,1.80cm) rectangle (2.25cm,3.30cm);
\draw[blue!70!black,thick] (0.75cm,1.80cm) rectangle (2.25cm,3.30cm);
\fill[black] (1.50cm,2.55cm) circle (1.8pt);
\node[below] at (1.50cm,2.55cm) {$\{n_{1}\alpha\}$};

\draw[<->] (0.75cm,1.55cm) -- (2.25cm,1.55cm);
\node[below] at (1.50cm,1.55cm) {\small $2r_{j}$};

\fill[red!18] (2.70cm,0.85cm) rectangle (4.00cm,2.35cm);
\draw[red!70!black,thick] (2.70cm,0.85cm) -- (4.00cm,0.85cm) -- (4.00cm,2.35cm) -- (2.70cm,2.35cm) -- cycle;
\fill[black] (3.45cm,1.60cm) circle (1.8pt);
\node[below] at (3.45cm,1.60cm) {$\{n_{2}\alpha\}$};

\fill[green!22] (2.70cm,2.70cm) rectangle (4.00cm,4.00cm);
\draw[green!50!black,thick] (2.70cm,2.70cm) -- (4.00cm,2.70cm) -- (4.00cm,4.00cm) -- (2.70cm,4.00cm) -- cycle;
\fill[black] (3.45cm,3.45cm) circle (1.8pt);
\node[below] at (3.45cm,3.45cm) {$\{n_{3}\alpha\}$};

\end{tikzpicture}
\caption{Diagram of the shrinking targets $E_{j}^{b}(\tau)$ in the case $d=2$. The black dots represent orbit points $\{n\alpha\}$ in $[0,1)^{2}$. Around each primitive orbit point we place the box $B_{j}(\{n\alpha\})$. }
\label{fig:targets2d}
\end{figure}
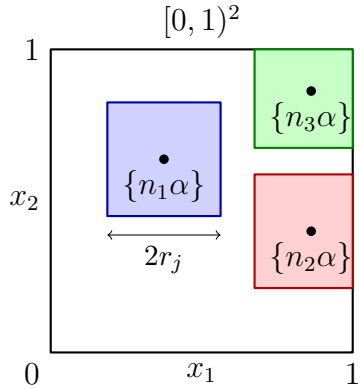
We first estimate the size of the targets.
\begin{lemma}\label{targetsize}
Fix $b \in \mathbb{Z}^{d}$ and $0 < \tau < c_{0}/8$. Then
\begin{equation*}
    \operatorname{Leb}(E_{j}^{b}(\tau)) = \delta_{d}(2\tau)^{d} + o(1), \quad j \to \infty,
\end{equation*}
where $\delta_{d}$ is as in \eqref{deltad}.
\end{lemma}

\begin{proof}
By \hyperref[packing]{Lemma~\ref{packing}}, the centers $\{n\alpha\}$, $1 \leqslant n \leqslant N_{j}$, are $c_{0}N_{j}^{-1/d}$-separated. Since
\begin{equation*}
    2r_{j} = 2\tau N_{j}^{-1/d} < \frac{c_{0}}{4}N_{j}^{-1/d},
\end{equation*}
the boxes $B_{j}(\{n\alpha\})$ are pairwise disjoint for $1\leqslant n\leqslant N_{j}$.

The only boxes whose volume may be smaller than $(2r_{j})^{d}$ are those whose centers are within distance $r_{j}$ of the boundary of $[0,1)^{d}$. The $r_{j}$-neighborhood of $\partial[0,1)^{d}$ can be covered by $O_{d}(r_{j}^{-(d-1)})$ axis-parallel cubes $Q$ of side length $r_{j}$. By \hyperref[packing]{Lemma~\ref{packing}}, 
\begin{equation*}
    \#\{1 \leqslant n \leqslant N_{j}: \{n\alpha\}\in Q\} \leqslant C_{\alpha,d} (1+r_{j}^{d} N_{j}) =O_{\alpha,d,\tau} (1).
\end{equation*}
Hence the number of boundary centers is
\begin{equation}\label{boundarycount}
    O_{\alpha,d,\tau}(r_{j}^{-(d-1)}) = O_{\alpha,d,\tau}(N_{j}^{(d-1)/d}) = o(N_{j}).
\end{equation}
Apply \hyperref[counting]{Theorem~\ref{counting}} with $A = [0,1)^{d}$, 
\begin{equation}\label{interiorcount}
    \#\big\{ 1 \leqslant n \leqslant N_{j} : n \in \mathcal{P}_{b} \big\} = (\delta_{d} + o(1))N_{j}.
\end{equation}
Therefore, combining \eqref{boundarycount} and \eqref{interiorcount} with the definition of $E_{j}^{b}(\tau)$ in \eqref{Ejdef} shows
\begin{equation*}
    \operatorname{Leb}(E_{j}^{b}(\tau)) = (\delta_{d} + o(1))N_{j}(2r_{j})^{d} + o(N_{j}r_{j}^{d}).
\end{equation*}
Since $N_{j}r_{j}^{d} = \tau^{d}$ by \eqref{rjdef}, this gives
\begin{equation*}
    \operatorname{Leb}(E_{j}^{b}(\tau)) = \delta_{d}(2\tau)^{d} + o(1).
\end{equation*}
\end{proof}

The following lemma reveals the quasi-independence between distinct scales, which is a higher-dimensional generalization of \cite{MR3996322}*{Theorem 3.4}.
\begin{lemma}\label{quasi}
Fix $b \in \mathbb{Z}^{d}$, $0 < \tau < c_{0}/8$, and $i \geqslant 1$. Then
\begin{equation}\label{quasiformula}
    \operatorname{Leb}(E_{i}^{b}(\tau) \cap E_{j}^{b}(\tau)) = \operatorname{Leb}(E_{i}^{b}(\tau))\operatorname{Leb}(E_{j}^{b}(\tau)) + o(1), \quad j \to \infty.
\end{equation}
\end{lemma}

\begin{proof}
Let
\begin{equation*}
    G = E_{i}^{b}(\tau).
\end{equation*}
This is a finite union of axis-parallel boxes, hence it is Jordan measurable.

Let $\partial_{*}G= (\partial G)\cup  \partial[0,1)^{d}$. The $r_{j}$-neighborhood of $\partial_{*}G$ can be covered by $O_{G}(r_{j}^{-(d-1)})$ axis-parallel cubes of side length $r_{j}$. By \hyperref[packing]{Lemma~\ref{packing}}, each such cube contains $O_{\alpha,d,\tau}(1)$ centers $\{n\alpha\}$, $1 \leqslant n \leqslant N_{j}$. Thus, by the same argument as in \eqref{boundarycount}, 
\begin{equation}\label{nearboundary}
    \#\big\{ 1 \leqslant n \leqslant N_{j} : \operatorname{dist}_{\infty}(\{n\alpha\}, \partial_{*}G) < r_{j} \big\} = o(N_{j}).
\end{equation}
If $\{n\alpha\} \in G$ and $\operatorname{dist}_{\infty}(\{n\alpha\}, \partial_{*}G) \geqslant r_{j}$, then
\begin{equation*}
    B_{j}(\{n\alpha\}) \subseteq G.
\end{equation*}
If $\{n\alpha\} \notin G$ and $\operatorname{dist}_{\infty}(\{n\alpha\}, \partial_{*}G) \geqslant r_{j}$, then
\begin{equation*}
    B_{j}(\{n\alpha\}) \cap G = \emptyset.
\end{equation*}
The centers counted in \eqref{nearboundary} contribute $o(N_{j}r_{j}^{d}) = o(1)$ to the total measure. Therefore
\begin{equation}\label{intermediateintersection}
    \operatorname{Leb}(G \cap E_{j}^{b}(\tau)) = (2r_{j})^{d} \#\big\{ 1 \leqslant n \leqslant N_{j} : n \in \mathcal{P}_{b}, \ \{n\alpha\} \in G \big\} + o(1).
\end{equation}
By \hyperref[counting]{Theorem~\ref{counting}}, applied to the Jordan measurable set $G$,
\begin{equation*}
    \#\big\{ 1 \leqslant n \leqslant N_{j} : n \in \mathcal{P}_{b}, \ \{n\alpha\} \in G \big\} = (\delta_{d}\operatorname{Leb}(G) + o(1))N_{j}.
\end{equation*}
Since $r_{j}^{d} = \tau^{d}N_{j}^{-1}$, \eqref{intermediateintersection} gives
\begin{equation*}
    \operatorname{Leb}(G \cap E_{j}^{b}(\tau)) = \delta_{d}(2\tau)^{d}\operatorname{Leb}(G) + o(1).
\end{equation*}
By \hyperref[targetsize]{Lemma~\ref{targetsize}},
\begin{equation*}
    \operatorname{Leb}(E_{j}^{b}(\tau)) = \delta_{d}(2\tau)^{d} + o(1).
\end{equation*}
Since $G = E_{i}^{b}(\tau)$, we obtain \eqref{quasiformula}.
\end{proof}

Finally, we show that $\limsup_{j\to\infty} E_{j}^{b}(\tau)$ has full Lebesgue measure.
\begin{lemma}\label{fulllimsup}
Fix $b \in \mathbb{Z}^{d}$ and $0 < \tau < c_{0}/8$. Then
\begin{equation}\label{fulllimsupformula}
    \operatorname{Leb}\bigg( \limsup_{j\to\infty} E_{j}^{b}(\tau) \bigg) = 1.
\end{equation}
\end{lemma}

\begin{proof}
Put
\begin{equation*}
    c_{\tau} = \delta_{d}(2\tau)^{d} > 0.
\end{equation*}
By \hyperref[targetsize]{Lemma~\ref{targetsize}},
\begin{equation}\label{targetlimit}
    \operatorname{Leb}(E_{j}^{b}(\tau)) = c_{\tau}+o(1)\qquad \text{as}\ j\to\infty.
\end{equation}
By \hyperref[quasi]{Lemma~\ref{quasi}}, for each fixed $i$,
\begin{equation}\label{fixediquasi}
    \operatorname{Leb}(E_{i}^{b}(\tau) \cap E_{j}^{b}(\tau)) = \operatorname{Leb}(E_{i}^{b}(\tau))\operatorname{Leb}(E_{j}^{b}(\tau)) + o(1)
\end{equation}
as $j \to \infty$.

Choose a subsequence $j_{\ell} \to \infty$ inductively so that, with
\begin{equation*}
    F_{\ell} = E_{j_{\ell}}^{b}(\tau),
\end{equation*}
we have
\begin{equation}\label{measureclose}
    \big| \operatorname{Leb}(F_{\ell}) - c_{\tau} \big| \leqslant 2^{-\ell}
\end{equation}
and, for every $1 \leqslant a < \ell$,
\begin{equation}\label{pairclose}
    \Big| \operatorname{Leb}(F_{a} \cap F_{\ell}) - \operatorname{Leb}(F_{a})\operatorname{Leb}(F_{\ell}) \Big| \leqslant 2^{-\ell}.
\end{equation}
This is possible by \eqref{targetlimit} and \eqref{fixediquasi}.

Let
\begin{equation*}
    S_{L} = \sum_{\ell=1}^{L} \operatorname{Leb}(F_{\ell}).
\end{equation*}
Then by \eqref{measureclose},
\begin{equation}\label{SLsim}
    S_{L} = c_{\tau} L + O(1).
\end{equation}
Furthermore,
\begin{equation*}
    \sum_{a,\ell=1}^{L} \operatorname{Leb}(F_{a} \cap F_{\ell}) = \sum_{\ell=1}^{L} \operatorname{Leb}(F_{\ell}) + 2\sum_{1 \leqslant a < \ell \leqslant L} \operatorname{Leb}(F_{a} \cap F_{\ell}).
\end{equation*}
Using \eqref{pairclose},
\begin{equation*}
    \sum_{a,\ell=1}^{L} \operatorname{Leb}(F_{a} \cap F_{\ell}) = \sum_{\ell=1}^{L} \operatorname{Leb}(F_{\ell}) + 2\sum_{1 \leqslant a < \ell \leqslant L} \operatorname{Leb}(F_{a})\operatorname{Leb}(F_{\ell}) + O\bigg( \sum_{\ell=1}^{\infty} \ell 2^{-\ell} \bigg).
\end{equation*}
Since
\begin{equation*}
    S_{L}^{2}=\sum_{\ell=1}^{L} \operatorname{Leb}(F_{\ell})^{2} + 2\sum_{1 \leqslant a < \ell \leqslant L} \operatorname{Leb}(F_{a})\operatorname{Leb}(F_{\ell})
\end{equation*}
we have
\begin{equation*}
    \sum_{\ell=1}^{L} \operatorname{Leb}(F_{\ell}) + 2\sum_{1 \leqslant a < \ell \leqslant L} \operatorname{Leb}(F_{a})\operatorname{Leb}(F_{\ell}) = S_{L}^{2} + \sum_{\ell=1}^{L} \operatorname{Leb}(F_{\ell})\big( 1 - \operatorname{Leb}(F_{\ell}) \big),
\end{equation*}
we get
\begin{equation*}
    \sum_{a,\ell=1}^{L} \operatorname{Leb}(F_{a} \cap F_{\ell}) = S_{L}^{2} + O(L).
\end{equation*}
Because $S_{L} \sim c_{\tau}L$ according to \eqref{SLsim}, we have $O(L)=o(S_{L}^{2})$. Therefore,
\begin{equation*}
    \sum_{a,\ell=1}^{L} \operatorname{Leb}(F_{a} \cap F_{\ell}) = (1 + o(1))S_{L}^{2}.
\end{equation*}
By \hyperref[SMBC]{Lemma~\ref{SMBC}}, we have
\begin{equation*}
    \operatorname{Leb}\bigg( \limsup_{\ell\to\infty} F_{\ell} \bigg) = 1.
\end{equation*}
Since $j_{\ell}$ is a subsequence of $j$, it is obvious that
\begin{equation*}
    \limsup_{\ell\to\infty} F_{\ell} \subseteq \limsup_{j\to\infty} E_{j}^{b}(\tau),
\end{equation*}
we obtain \eqref{fulllimsupformula}.
\end{proof}

\section{Proof of main theorem}

We first prove the result on a fixed fundamental cube.

\begin{proposition}\label{fixedbprop}
Let $\alpha\in \Lambda^{d}$ with $d\geqslant 2$. Fix $b \in \mathbb{Z}^{d}$.  Then for almost every $\eta \in [0,1)^{d}$,
\begin{equation}\label{fixedbclaim}
    \liminf_{n\to\infty} n^{1/d} \min_{\substack{m\in\mathbb{Z}^{d}\\ \gcd(n,m_{1},\ldots,m_{d})=1}} \|b + \eta - n\alpha + m\|_{\infty} = 0.
\end{equation}
\end{proposition}

\begin{proof}
Let $\tau_{k} \downarrow 0$ be a sequence with $0 < \tau_{k} < c_{0}/8$ for each $k\geqslant 1$. By \hyperref[fulllimsup]{Lemma~\ref{fulllimsup}}, for each $k$,
\begin{equation*}
    \operatorname{Leb}\bigg( \limsup_{j\to\infty} E_{j}^{b}(\tau_{k}) \bigg) = 1.
\end{equation*}
Therefore the set
\begin{equation*}
    \Omega_{b}' = \bigcap_{k=1}^{\infty} \bigg(\limsup_{j\to\infty} E_{j}^{b}(\tau_{k})\bigg)
\end{equation*}
has full measure in $[0,1)^{d}$.

Moreover, remove the countable orbit
\begin{equation*}
    \operatorname{Orb}_{\alpha} = \big\{ \{n\alpha\} : n \geqslant 1 \big\}.
\end{equation*}
The set
\begin{equation*}
    \Omega_{b} = \Omega_{b}' \setminus \operatorname{Orb}_{\alpha}
\end{equation*}
still has full measure.

Fix $\eta \in \Omega_{b}$. Let $k \geqslant 1$. Since $\eta \in \limsup_{j\to\infty} E_{j}^{b}(\tau_{k})$, there exist infinitely many $j$ such that $\eta \in E_{j}^{b}(\tau_{k})$. For each such $j$, there is $1 \leqslant n_{j} \leqslant N_{j}$ satisfying
\begin{equation}\label{njprimitive}
    n_{j} \in \mathcal{P}_{b}
\end{equation}
and
\begin{equation}\label{hitestimate}
    \|\eta - \{n_{j}\alpha\}\|_{\infty} < \tau_{k} N_{j}^{-1/d}.
\end{equation}

The integers $n_{j}$ occurring for infinitely many $j$'s must be unbounded. Indeed, if a fixed $n$ occurred along infinitely many $j$, then $N_{j} \to \infty$ and \eqref{hitestimate} would force $\eta = \{n\alpha\}$, contrary to $\eta \notin \operatorname{Orb}_{\alpha}$. Thus we may select a subsequence with $n_{j} \to \infty$.

For each selected $j$, define
\begin{equation*}
    m_{j} = \big( \lfloor n_{j}\alpha_{1}\rfloor - b_{1}, \ldots, \lfloor n_{j}\alpha_{d}\rfloor - b_{d} \big) \in \mathbb{Z}^{d}.
\end{equation*}
By \eqref{njprimitive} and the definition of $\mathcal{P}_{b}$,
\begin{equation}\label{mjprimitive}
    \gcd(n_{j}, m_{j,1}, \ldots, m_{j,d}) = 1.
\end{equation}
Moreover,
\begin{equation}\label{identity}
    b + \eta - n_{j}\alpha + m_{j} = \eta - \{n_{j}\alpha\}.
\end{equation}
Combining \eqref{hitestimate} and \eqref{identity},
\begin{equation*}
    \|b + \eta - n_{j}\alpha + m_{j}\|_{\infty} < \tau_{k} N_{j}^{-1/d}.
\end{equation*}
Since $n_{j} \leqslant N_{j}$,
\begin{equation*}
    n_{j}^{1/d} \|b + \eta - n_{j}\alpha + m_{j}\|_{\infty} < \tau_{k} \bigg(\frac{n_{j}}{N_{j}}\bigg)^{1/d} \leqslant \tau_{k}.
\end{equation*}
Using \eqref{mjprimitive}, we obtain
\begin{equation*}
    \liminf_{n\to\infty} n^{1/d} \min_{\substack{m\in\mathbb{Z}^{d}\\ \gcd(n,m_{1},\ldots,m_{d})=1}} \|b + \eta - n\alpha + m\|_{\infty} \leqslant \tau_{k}.
\end{equation*}
This holds for every $k$. Letting $k \to \infty$ gives \eqref{fixedbclaim}.
\end{proof}

\begin{proof}[Proof of Theorem \ref{mainthm}]
For each $b \in \mathbb{Z}^{d}$, \hyperref[fixedbprop]{Proposition~\ref{fixedbprop}} gives a full-measure subset $\Omega_{b} \subseteq [0,1)^{d}$ such that \eqref{fixedbclaim} holds for every $\eta \in \Omega_{b}$.

Define
\begin{equation*}
    \Omega = \bigcup_{b\in\mathbb{Z}^{d}} (b + \Omega_{b}) \subseteq \mathbb{R}^{d}.
\end{equation*}
Since $\mathbb{R}^{d}$ is the countable union of the translates $b + [0,1)^{d}$, up to boundary sets of measure zero, the set $\Omega$ has full Lebesgue measure in $\mathbb{R}^{d}$. If $\gamma \in \Omega$, then $\gamma = b + \eta$ for some $b \in \mathbb{Z}^{d}$ and $\eta \in \Omega_{b}$. \hyperref[fixedbprop]{Proposition~\ref{fixedbprop}} gives exactly \eqref{mainclaim}. This proves the theorem.
\end{proof}

\appendix
\section{Full measure of non-singular set}

\begin{theorem}\label{fullLebLambda}
    The subset $\Lambda^{d}\subseteq \mathbb{R}^{d}$ has full Lebesgue measure.
\end{theorem}

\begin{proof}
It suffices to work on $\mathbb{T}^{d}$, identified with $[0,1)^{d}$.

For $N\geqslant1$, define
\begin{equation*}
    \psi_{\alpha}(N)\coloneq \min_{1\leqslant q\leqslant N}\|q\alpha\|_{\mathbb{T}^{d},\infty}.
\end{equation*}
We estimate the measure of
\begin{equation*}
    \Sigma\coloneq \bigg\{\alpha\in\mathbb{T}^{d}:\limsup_{N\to\infty}N^{1/d}\psi_{\alpha}(N)=0\bigg\}.
\end{equation*}
For $\varepsilon>0$, set
\begin{equation*}
    A_{N}(\varepsilon)\coloneq \big\{\alpha\in\mathbb{T}^{d}:\exists\,1\leqslant q\leqslant N\text{ such that }\|q\alpha\|_{\mathbb{T}^{d},\infty}<\varepsilon N^{-1/d}\big\}.
\end{equation*}
For fixed $q$, the map $\alpha\mapsto q\alpha\bmod\mathbb{Z}^{d}$ preserves Haar measure on $\mathbb{T}^{d}$. Hence, for all sufficiently large $N$,
\begin{equation*}
    \operatorname{Leb}\big(\{\alpha\in\mathbb{T}^{d}:\|q\alpha\|_{\mathbb{T}^{d},\infty}<\varepsilon N^{-1/d}\}\big)=(2\varepsilon)^{d}N^{-1}.
\end{equation*}
Therefore,
\begin{equation*}
    \operatorname{Leb}(A_{N}(\varepsilon))\leqslant \sum_{q=1}^{N}(2\varepsilon)^{d}N^{-1}=(2\varepsilon)^{d}.
\end{equation*}
Let
\begin{equation*}
    S(\varepsilon)\coloneq \big\{\alpha\in\mathbb{T}^{d}:\psi_{\alpha}(N)<\varepsilon N^{-1/d}\text{ for all sufficiently large }N\big\}.
\end{equation*}
Then
\begin{equation*}
    S(\varepsilon)\subseteq \liminf_{N\to\infty}A_{N}(\varepsilon).
\end{equation*}
By Fatou's Lemma, we obtain
\begin{equation*}
    \operatorname{Leb}(S(\varepsilon))\leqslant \operatorname{Leb}\bigg(\liminf_{N\to\infty}A_{N}(\varepsilon)\bigg)\leqslant \liminf_{N\to\infty}\operatorname{Leb}(A_{N}(\varepsilon))\leqslant (2\varepsilon)^{d}.
\end{equation*}
Since
\begin{equation*}
    \Sigma=\bigcap_{k=1}^{\infty}S(1/k),
\end{equation*}
we have, for every $k\geqslant1$,
\begin{equation*}
    \operatorname{Leb}(\Sigma)\leqslant \operatorname{Leb}(S(1/k))\leqslant (2/k)^{d}.
\end{equation*}
Letting $k\to\infty$, we get
\begin{equation*}
    \operatorname{Leb}(\Sigma)=0.
\end{equation*}

It remains to impose rational independence. The set of rationally dependent $\alpha$'s is contained in the countable union
\begin{equation*}
    \bigcup_{\substack{k_{0}\in\mathbb{Z}\\ (k_{1},\ldots,k_{d})\in\mathbb{Z}^{d}\setminus\{0\}}}\big\{\alpha\in\mathbb{R}^{d}:k_{0}+k_{1}\alpha_{1}+\cdots+k_{d}\alpha_{d}=0\big\}.
\end{equation*}
This is a countable union of proper affine hyperplanes and hence has Lebesgue measure zero. Combining this with $\operatorname{Leb}(\Sigma)=0$, and then extending periodically from $[0,1)^{d}$ to $\mathbb{R}^{d}$, proves that $\Lambda^{d}$ has full Lebesgue measure.
\end{proof}

\section{Proof of \hyperref[SMBC]{Lemma~\ref{SMBC}}}
\label{ABC}
\begin{proof}
Fix $K\geqslant1$. For $L\geqslant K$, define
\begin{equation*}
    S_{K,L}\coloneq \sum_{\ell=K}^{L}\mu(F_{\ell}),\qquad I_{K,L}\coloneq \sum_{a,\ell=K}^{L}\mu(F_{a}\cap F_{\ell}).
\end{equation*}
Since $K$ is fixed and $S_L\to\infty$, we have
\begin{equation*}
    S_{K,L}=S_L-\sum_{\ell=1}^{K-1}\mu(F_{\ell})=S_L+O_K(1)\sim S_L.
\end{equation*}
Moreover, a direct calculation implies that
\begin{equation*}
    0\leqslant \sum_{a,\ell=1}^{L}\mu(F_{a}\cap F_{\ell})-I_{K,L}\leqslant 2\sum_{a=1}^{K-1}\sum_{\ell=1}^{L}\mu(F_{a}\cap F_{\ell})\leqslant 2KS_L=o(S_L^2).
\end{equation*}
Therefore, combining the above inequality with \eqref{SMBCcondition} and $S_{K,L}\sim S_L$,
\begin{equation}\label{IS}
    I_{K,L}=(1+o(1))S_{K,L}^{2}.
\end{equation}

Now set
\begin{equation*}
    Z_{K,L}(x)\coloneq \sum_{\ell=K}^{L}\mathbf{1}_{F_{\ell}}(x).
\end{equation*}
Then
\begin{equation*}
    \int_X Z_{K,L}(x)\,\mathrm{d}\mu(x)=S_{K,L},\qquad \int_X Z_{K,L}(x)^2\,\mathrm{d}\mu(x)=I_{K,L}.
\end{equation*}
By the Cauchy--Schwarz inequality,
\begin{equation*}
    S_{K,L}=\int_{\bigcup_{\ell=K}^{L}F_{\ell}}Z_{K,L}(x)\,\mathrm{d}\mu(x)\leqslant \mu\bigg(\bigcup_{\ell=K}^{L}F_{\ell}\bigg)^{1/2}I_{K,L}^{1/2}.
\end{equation*}
Thus by \eqref{IS},
\begin{equation*}
    \mu\bigg(\bigcup_{\ell=K}^{L}F_{\ell}\bigg)\geqslant \frac{S_{K,L}^{2}}{I_{K,L}}=1-o(1).
\end{equation*}
Letting $L\to\infty$, we obtain
\begin{equation*}
    \mu\bigg(\bigcup_{\ell=K}^{\infty}F_{\ell}\bigg)=1.
\end{equation*}
Since $K\geqslant1$ is arbitrary,
\begin{equation*}
    \mu\bigg(\limsup_{\ell\to\infty}F_{\ell}\bigg)=\mu\bigg(\bigcap_{K=1}^{\infty}\bigcup_{\ell=K}^{\infty}F_{\ell}\bigg)=1.
\end{equation*}
The proof is complete.
\end{proof}

\section*{Acknowledgments}
X. Wang thanks Prof. Wencai Liu for useful discussion. This work was supported by NSF DMS-2246031
and NSF DMS-2052572.

\bibliography{main}
\end{document}